\documentclass[a4paper, 11pt]{amsart}
\usepackage{amsmath,amsthm,amssymb}
\usepackage[a4paper,margin=1.12in]{geometry}
\usepackage{mathtools}
\usepackage{enumerate}
\usepackage{ulem}
\usepackage{cite}
\usepackage[driverfallback=dvipdfm]{hyperref}
\usepackage[dvipdfmx]{graphicx}
\usepackage{bbold}
\usepackage{xcolor}

\usepackage{setspace}
\usepackage{dsfont}
\usepackage{comment}
\usepackage{tikz-cd}

\DeclareMathOperator*{\esssup}{ess\,sup}

\theoremstyle{plain}
\newtheorem{thm}{Theorem}[section]
\newtheorem{prop}[thm]{Proposition}
\newtheorem{lem}[thm]{Lemma}

\newtheorem{rmk}[thm]{Remark}

\theoremstyle{definition}
\newtheorem{dfn}{Definition}[section]
\newtheorem{prob}{Problem}

\makeatletter
	
	\@addtoreset{equation}{section}
\makeatother

\title[Isometries on Sobolev spaces]{Metric Rigidity in Anchored Sobolev Spaces on Intervals}

\author[M.--R. Lin]{Min--Ruei Lin}
\address[M.--R. Lin]
{Department of Applied Mathematics,
National Sun Yat--sen University,
Kaohsiung, 80424, Taiwan;
and the
Graduate School of Science and Technology, Niigata University, Niigata 950--2181, Japan}
\email{m082030021@student.nsysu.edu.tw}

\subjclass[2020]{Primary 46B04; Secondary 46B40, 46E35.}

\keywords{Tingley’s problem, positive unit sphere, surjective isometry, Sobolev space, ordered Banach space, phase-isometry.}

\begin{document}

\begin{abstract}
    For $1\le p\le\infty$ and $i=1,2$, let $W^{k_i,p}(\Omega_i)$ be the Sobolev space on a bounded open interval $\Omega_i$ with differentiability order $k_i$.
    We equip $W^{k_i,p}(\Omega_i)$ with an anchored Sobolev norm and the order $\ge_{k_i,p}$ defined by $f^{(j)}(x_i)\ge 0$ for each $j=0,\ldots,k_i-1$ and $f^{(k_i)}\ge 0$ a.e.
    We show that the positive unit spheres of $W^{k_1,p}(\Omega_1)$ and $W^{k_2,p}(\Omega_2)$ are surjectively isometric if and only if $k_1=k_2$.
    Every such isometry extends uniquely to a complex--linear isometric order isomorphism, for which we obtain a coordinate representation.
    The same conclusions hold for surjective phase--isometries.
    For $1<p<\infty$, they also hold for surjective norm--additive maps.
\end{abstract}
\maketitle

\section{Introduction}
Tingley~\cite{T87} asks whether, for arbitrary Banach spaces E and F, each surjective isometry $T:S(E)\to S(F)$ is the restriction of a unique surjective real--linear isometry from $E$ onto $F$.
This extension problem has been studied extensively in several classes of function spaces~\cite{cue1, cue2, hat1, hat2, hat3, wang1, wang2, wang3, tan1, tan2}.
In the particular case of $L^p(\mu)$--spaces, the answer is affirmative whenever the measure $\mu$ is $\sigma$--finite~\cite{tan1,tan2}.

A promising direction is to consider a variant of Tingley’s problem in which the entire unit sphere is replaced by its positive part.
This viewpoint is natural for ordered Banach spaces whose positive cones generate their underlying real parts; in the complex setting, the whole space is then recovered by complexification.
Peralta~\cite{PAM} introduced such a variant.
A more general form of his question is stated below:
\begin{prob}\label{Tingley}
    Let $T:S(E)^+\to S(F)^+$ be a surjective isometry between the positive unit spheres of two ordered Banach spaces $E$ and $F$ with generating cones; namely
    $$\|T(x)-T(y)\| = \|x-y\|\qquad (x,y\in S(E)^+).$$
    Is there a surjective linear isometry $\Phi: E\to F$ extending $T$?
\end{prob}
For complex $L^p$--spaces with $1\le p\le\infty$, the answer is affirmative~\cite{LNW21}.
In 2026, Enami \textit{et al.}~\cite{C1LIP} treated $C^1_{\mathbb{C}}([0,1])$ and $\operatorname{Lip}_{\mathbb{C}}([0,1])$ with anchored norms, while observing that their method does not extend to higher--order differentiable spaces.
A different approach appears in Lin's study of $AC^p(\Omega)$~\cite{L26}.
In the case of the norm $|f(x_0)|+\|f'\|_1$ with a fixed base point $x_0$, the space $AC^1(\Omega)$ is identified with an $L^1$--space by adjoining an atomic coordinate.
Since the space $AC^1(\Omega)$, equipped with this norm and order, can be canonically identified with the Sobolev space $W^{1,1}(\Omega)$, it is natural to consider the more general spaces $W^{k,p}(\Omega)$ with $k\in\mathbb{N}$ and $1\le p\le\infty$.
In the present paper, we extend this atom-adjoining construction for higher--order derivatives: $W^{k,p}(\Omega)$ is represented as an $L^p$-space with finitely many atomic coordinates.
The corresponding extension results then follow from the theorems of Leung, Ng, and Wong.

Before addressing the positive-sphere variant of Tingley’s problem for $W^{k,p}(\Omega)$, one must first determine whether a surjective isometry can exist between spaces of different differentiability orders.
Related rigidity phenomena have been studied for spaces of continuously differentiable functions.
Araujo~\cite{A04} proved that the existence of a linear biseparating map between $C^r$-- and $C^s$--spaces forces $r=s$, while Leung, Ng, and Tang~\cite{LNT22} obtained the analogous conclusion for surjective linear isometries between broad classes of vector-valued differentiable function spaces.
Motivated by these results, we show that the differentiability order of the ordered Sobolev space $W^{k,p}$ is determined by the metric structure of its positive unit sphere, without assuming linearity of the given isometry.

The purpose of this paper is to provide an affirmative answer to Problem~\ref{Tingley} for the Sobolev spaces $W^{k,p}(\Omega)$ on a bounded open interval $\Omega$, equipped with an anchored $W^{k,p}$--norm.
The paper is organized as follows.
In Section~2, we first identify $W^{k,p}(\Omega)$ with an appropriate $L^p(X,\mu)$--space via a complex--linear isometric order isomorphism and establish the corresponding extension theorem.
Finally, we establish an order--geometric rigidity result.
The existence of a surjective isometry or phase--isometry between the positive unit spheres of $W^{k_1,p}(\Omega_1)$ and $W^{k_2,p}(\Omega_2)$ forces $k_1=k_2$, and every such map extends uniquely to a complex--linear isometric order isomorphism; for $1<p<\infty$, the same conclusions hold for surjective norm--additive maps.

\section{Isometries on the positive unit sphere of Sobolev spaces}\label{sec2}
Let $\Omega$ be a bounded open interval in $\mathbb{R}$ equipped with Lebesgue measure $m$.
For each $1\le p\le \infty$ and $k\in\mathbb{N}$, we consider the complex vector space of $k$--times weakly differentiable functions 
$$W^{k,p}_\mathbb{C}(\Omega):= \left\{ f\in L^p(\Omega): \exists  g_j\in L^p(\Omega) \text{ s.t. } \int_\Omega f\varphi^{(j)} = (-1)^j\int_\Omega g_j\varphi,\ \forall \varphi\in C_c^\infty(\Omega) \right\},$$
where $g_j$ is the \textit{$j$--th weak derivative} of $f$, for each $j=1,\ldots,k$.
We write $f^{(j)} = g_j$ for convenience, and $f^{(0)} = f$.
In this paper, for each $j=0,\ldots,k-1$, we identify $f^{(j)}\in L^p(\Omega)$ with its absolutely continuous representative, and hence the values $f^{(j)}(x_0)$ are well defined.
Let $W^{k,p}_\mathbb{R}(\Omega)$ be the real subspace of $W^{k,p}_\mathbb{C}(\Omega)$.
Fix a base point $x_0\in\Omega$.
We define the positive cone of vector space $W^{k,p}_\mathbb{C}(\Omega)$ as 
$$W^{k,p}_\mathbb{C}(\Omega)^+:=\{f\in W^{k,p}_\mathbb{R}(\Omega): f^{(k)}\ge 0\ \mathrm{a.e.},\ f^{(j)}(x_0)\ge 0,\ \forall j=0,\cdots, k-1\} = W^{k,p}_\mathbb{R}(\Omega)^+,$$
and its induced order $\ge_{k,p}$.
Let the anchored $W^{k,p}$--norm be defined as 
\begin{equation}\label{W_norm}
\|f\|_{k,p} :=
\begin{cases}
    \Big(\displaystyle\sum_{j=0}^{k-1}|f^{(j)}(x_0)|^p + \|f^{(k)}\|_p^p\Big)^{1/p}\qquad &(1\le p<\infty)\\
    \max\big\{\displaystyle\max_{0\le j<k}\{|f^{(j)}(x_0)|\},\ \|f^{(k)}\|_\infty\big\}\qquad &(p=\infty)
\end{cases},
\end{equation}
where $x_0\in\Omega$ is the base point.
Then we denote the order Banach space $(W^{k,p}_\mathbb{C}(\Omega),\ \|\cdot\|_{k,p},\ \ge_{k,p})$ by $W^{k,p}(\Omega)$.
Let $S(W^{k,p}(\Omega))$ denote the unit sphere of $W^{k,p}(\Omega)$, and $W^{k,p}(\Omega)^+$ denote the positive cone of $W^{k,p}(\Omega)$.
Define
$$S(W^{k,p}(\Omega))^+:=S(W^{k,p}(\Omega))\cap W^{k,p}(\Omega)^+=\{f\in W^{k,p}(\Omega)^+: \|f\|_{k,p}=1\},$$
the positive part of the unit sphere of $W^{k,p}(\Omega)$.
The positive unit sphere of $L^p(\Omega)$ is written as 
$$S(L^p(\Omega))^+=\{u\in L^p(\Omega): \|u\|_p =1, u\ge 0\ \mathrm{a.e.}\}$$
with the usual order and the $L^p$--norm.

We now define several types of maps.
\begin{dfn}\label{dfn:2.1}
    Let $A, B$ be non-empty subsets of normed spaces $E,F$, respectively.
    A map $T:A\to B$ is called 
    \begin{enumerate}
    \item a \textit{phase--isometry} if 
    $$\big\{\|x-\alpha y\|_E:\alpha\in\mathbb{T}\big\} = \{\|T(x)-\alpha T(y)\|_F:\alpha\in\mathbb{T}\}\qquad (x,y\in A),$$
    where $\mathbb{T}:=\{z\in\mathbb{C}: |z|=1\}$.
    \item a \textit{norm--additive map} if 
    $$\|x+y\|_E = \|T(x)+T(y)\|_F\qquad (x,y \in A).$$
    \end{enumerate}
    If, moreover, $E$ and $F$ are ordered Banach spaces, then
    \begin{enumerate}
        \item[(3)] a \textit{complex--linear isometric order isomorphism} $\Phi:E\to F$ is a complex--linear surjective isometry such that $\Phi(E^+)=F^+$ and $\Phi^{-1}(F^+) = E^+$, where $E^+$ and $F^+$ are the positive cones of $E$ and $F$, respectively.
    \end{enumerate}
\end{dfn}
Phase--isometries have been investigated in~\cite{HLW24,LHW24,THH26}, whereas norm-additive maps have been studied in~\cite{HDL22, SS24, ZTW19, ZTW21}.

For each $f \in W^{k,p}(\Omega)$, the highest-order weak derivative $f^{(k)}$ belongs to $L^{p}(\Omega)$. Moreover, upon identifying the lower-order derivatives with their absolutely continuous representatives, the quantities $f^{(j)}(x_{0})$, $j=0,\ldots,k-1$, are well-defined complex numbers.
This motivates adjoining $k$ isolated atoms to the measurable space $(\Omega, \mathcal{L}(\Omega),m)$ in order to encode these point-evaluation coordinates.
Let $\{\ast^{(j)}:j=0,\ldots,k-1\}$ be a set disjoint from $\Omega$.
We denote the Lebesgue $\sigma$--algebra of $\Omega$ by $\mathcal{L}(\Omega)$, and denote Dirac measure at $\ast^{(j)}$ by $\delta_{\ast^{(j)}}$.
In particular, the resulting representation is stated in the following proposition.

\begin{prop}\label{prop:2.1}
    Suppose $1\le p\le \infty$.
    Fix $k\in\mathbb{N}$.
    Let $X=\{\ast^{(j)}:j=0,\ldots,k-1\}\sqcup\Omega$.
    Define a $\sigma$-algebra $\mathfrak{A}:=\{A\subseteq X: A\cap\Omega\in\mathcal{L}(\Omega)\}$, and a measure $\mu(E):=m(E\cap\Omega) + \sum_{j=0}^{k-1}\delta_{\ast^{(j)}}(E)$ for each $E\in\mathfrak{A}$.
    Then $W^{k,p}(\Omega)$ is complex--linearly isometric and order isomorphic to $L^p(X,\mathfrak{A},\mu)$.
\end{prop}
\begin{proof}
We write $L^p(X,\mathfrak{A}, \mu)$ as $L^p(\mu)$ for simplicity.
Define a map
\begin{equation*}
J:W^{k,p}(\Omega)\to L^p(\mu)\qquad\mathrm{by}\qquad
J(f)=\sum_{j=0}^{k-1}f^{(j)}(x_0)\chi_{\{\ast^{(j)}\}}+f^{(k)}\chi_\Omega.
\end{equation*}
Linearity is immediate.
For $1\le p<\infty$, the definition of $\mu$ reveals that
$$
\|Jf\|_{L^p(\mu)}^p=\sum_{j=0}^{k-1}|(Jf)(*^{(j)})|^p+\int_{\Omega}|(Jf)(t)|^p\,dm=\sum_{j=0}^{k-1}|f^{(j)}(x_0)|^p+\|f^{(k)}\|_p^p=\|f\|_{k,p}^p,
$$
and hence, $\|Jf\|_{L^p(\mu)}=\|f\|_{k,p}$.
If $p=\infty$, then 
\begin{align*}
\|Jf\|_{L^\infty(\mu)}&:=\esssup_{t\in X}\left\{\big|\sum_{j=0}^{k-1}f^{(j)}(x_0)\chi_{\{\ast^{(j)}\}}(t) + f^{(k)}(t)\chi_{\Omega}(t)\big|\right\}\\
&=\max\left\{|f(x_0)|,\ldots,\ |f^{(k-1)}(x_0)|,\ \|f^{(k)}\|_{L^\infty(\Omega)}
\right\}=\|f\|_{k,\infty}.
\end{align*}
Thus, for each $1\le p\le\infty$, $J$ is an isometry.

We next prove that $J$ is surjective.
Let $h\in L^{p}(\mu)$, and set $a_j:=h\bigl(\ast^{(j)}\bigr)$ for each
$j=0,\ldots,k-1$, and $g:=h|_{\Omega}\in L^{p}(\Omega)$.
Since $\Omega$ is bounded, $L^{p}(\Omega)\subseteq L^{1}(\Omega)$ for every $1\leq p\leq\infty$.
Hence the function
$$
R_k g(x):= \frac{1}{(k-1)!}\int_{x_0}^{x}(x-t)^{k-1}g(t)\,dt,
\qquad x\in\Omega,
$$
is well defined.
In particular, the indefinite integral of an $L^{1}$-function is absolutely continuous and has that function as
its weak derivative.
Iterating this fact yields, for $j=0,\ldots,k-1$,
$$
D^{j}(R_k g)(x)=\frac{1}{(k-j-1)!}\int_{x_0}^{x}(x-t)^{k-j-1}g(t)\,dt,
$$
where the right-hand side denotes the absolutely continuous representative of $D^{j}(R_k g)$, and $D^{k}(R_k g)=g$ a.e. on $\Omega$.
In particular, $D^{j}(R_k g)(x_0)=0$ for each $j=0,\ldots,k-1$.
Moreover, since $\Omega$ is bounded, each of the above lower-order derivatives belongs to $L^{p}(\Omega)$.

Now define
\begin{equation}\label{eq:Sobolev}
f(x):=\sum_{\ell=0}^{k-1}a_\ell\frac{(x-x_0)^\ell}{\ell!}+R_k g(x),
\qquad x\in\Omega.
\end{equation}
The polynomial term belongs to $W^{k,p}(\Omega)$, and therefore $f\in W^{k,p}(\Omega)$. For every $j=0,\ldots,k-1$, using the absolutely continuous representative of $f^{(j)}$, we obtain
$$
f^{(j)}(x_0)=a_j,\qquad\text{and}\qquad f^{(k)}=g
\quad\text{a.e. on }\Omega.
$$
Consequently,
$$
Jf=\sum_{j=0}^{k-1}
a_j\chi_{\{\ast^{(j)}\}}+g\chi_{\Omega}=h.
$$
Thus $J$ is surjective.

Finally, we verify that $J$ is an order isomorphism.
Since $J$ is linear and bijective, it is sufficient to show $f\in W^{k,p}_\mathbb{R}(\Omega)^+$ if and only if $J(f)\in L^p(\mu)^+$.
If $f\in W^{k,p}_\mathbb{R}(\Omega)^+$, we have
\begin{align*}
f\ge_{k,p}0\qquad&\Longleftrightarrow\qquad
f^{(j)}(x_0)\ge 0,\quad j=0,\ldots,k-1,
\quad\mathrm{and}\quad
f^{(k)}\ge 0\quad \text{a.e. on } \Omega\\
&\Longleftrightarrow\qquad \sum_{j=0}^{k-1}f^{(j)}(x_0)\chi_{\{\ast^{(j)}\}}+f^{(k)}\chi_\Omega\ge 0\quad\mu\mathrm{-a.e.}
\end{align*}
That is, $J(f)\ge 0$ $\mu$--a.e.
Equivalently, $J(f)\in L^p(\mu)^+$.
Therefore $J$ is a complex--linear isometric order isomorphism.
\end{proof}

Recall that a nonzero element $u$ of the positive cone of a vector lattice is a \textit{lattice atom} if
$$
 0\le v\le u \quad\Longrightarrow\quad v=\lambda u
 \quad\text{for some }\lambda\in[0,1].
$$
An order isomorphism maps lattice atoms bijectively onto lattice atoms.
The following proposition identifies the lattice atoms of the $L^p$--space associated with the measure space introduced in Proposition~\ref{prop:2.1}.

\begin{prop}\label{prop:2.2}
Let $1\leq p\leq\infty$, and let $(X,\mathfrak{A},\mu)$ be the measure space defined in Proposition~\ref{prop:2.1}.
Then the lattice atoms of $L^p(\mu)$ are precisely $\big\{
c\chi_{\{\ast^{(j)}\}}: c>0,\ j=0,\ldots,k-1
\big\}$.
Consequently, the positive norm-one lattice atoms of $L^p(\mu)$ are exactly $\chi_{\{\ast^{(0)}\}},\ldots,\chi_{\{\ast^{(k-1)}\}}$.
\end{prop}
\begin{proof}
For every $c>0$ and $j=0,\ldots,k-1$, the function
$c\chi_{\{\ast^{(j)}\}}$ is a lattice atom.
To verify this, if $0 \leq v \leq c\chi_{\{\ast^{(j)}\}}$, then $v=0$ $\mu\text{--a.e.}$ outside $\{\ast^{(j)}\}$, and hence $v=\lambda c\chi_{\{\ast^{(j)}\}}$ for some $\lambda\in[0,1]$.

Conversely, let $u\in L^p(\mu)^+$ be a lattice atom, and set $A:=\{x\in X:u(x)>0\}$.
If $B\subseteq A$ is measurable, then $0\leq u\chi_B\leq u$, so $u\chi_B=\lambda u$ for some $\lambda\in[0,1]$ by the definition of lattice atom.
Since $u>0$ on $A$, it follows that $\chi_B=\lambda$ $\mu\text{--a.e.}$ on $A$, and hence $\lambda\in\{0,1\}$.
Therefore, every measurable subset of $A$ is either null or co--null in $A$, and hence $A$ is a measure-theoretic atom.

Suppose that $\mu(A\cap\Omega)>0$.
Since the restriction of $\mu$ to $\Omega$ is the Lebesgue measure and is atomless, there exists a measurable set
$B\subseteq A\cap\Omega$ such that $0<\mu(B)<\mu(A\cap\Omega)$.
In particular, $0<\mu(B)<\mu(A)$, contradicting the fact that $A$ is an atom.
Therefore, $\mu(A\cap\Omega)=0$.
That is, $A$ cannot contain any portion of the atomless part $\Omega$.

Since $\mu(A)>0$, the set $A$ must contain at least one point among $\ast^{(0)},\ldots,\ast^{(k-1)}$.
It cannot contain two distinct such points, since one of the corresponding singletons would be a measurable subset of $A$ whose measure and whose complement in $A$ both have positive measure.
Hence, modulo a null set, $A=\{\ast^{(j)}\}$ for some $j\in\{0,\ldots,k-1\}$.
Consequently, $u=c\chi_{\{\ast^{(j)}\}}$ for some $c>0$.
Finally, since $\mu(\{\ast^{(j)}\})=1$, we have $\big\|c\chi_{\{\ast^{(j)}\}}\big\|_p=c$ for every $1\leq p\leq\infty$.
Therefore, the positive norm-one lattice atoms are precisely $\chi_{\{\ast^{(0)}\}},\ldots,\chi_{\{\ast^{(k-1)}\}}$.
\end{proof}

We use the following extension theorems for the positive unit spheres of $L^p$--spaces.

\begin{prop}[{\cite[Theorem 11(b) and Theorem 16]{LNW21}}]\label{prop:2.3}
    Let $p\in [1,\infty]$, and let $(\Gamma_1,\mathfrak{B}_1, \nu_1)$ and $(\Gamma_2, \mathfrak{B}_2, \nu_2)$ be arbitrary measure spaces.
    Suppose that $\Delta:S(L^p(\nu_1))^+\to S(L^p(\nu_2))^+$ is a surjective isometry.
    Then $\Delta$ can be extended (uniquely) to an isometric order isomorphism from $L^p(\nu_1)$ onto $L^p(\nu_2)$.
    When $p=\infty$, the extension is a $*$--isomorphism.
\end{prop}

\begin{rmk}\label{rmk:2.4}
    A map $\widetilde{\Delta}:L^\infty(\nu_1)\to L^\infty(\nu_2)$ is a $*$--isomorphism if and only if it is a complex--linear isometric order isomorphism.
\end{rmk}
\begin{proof}
Recall that a $*$--isomorphism between commutative $C^*$--algebras is a complex--linear isometric order isomorphism.
It remains to prove the non--trivial direction.

By the Gelfand transform, $\widetilde{\Delta}$ induces uniquely a complex--linear isometric order isomorphism $\widetilde{\Phi}:C(\mathcal{M}_1)\to C(\mathcal{M}_2)$, where $\mathcal{M}_i$ is the maximal ideal space of $L^\infty(\nu_i)$ for $i=1,2$.
Since $\widetilde{\Phi}$ is a complex--linear surjective isometry, by the Banach-Stone theorem for continuous function spaces
\begin{equation}\label{eq:BST}
\widetilde{\Phi}(f)(y) = h(y)f(\eta(y))\qquad (f\in C(\mathcal{M}_1)),
\end{equation}
where $\eta:\mathcal{M}_2\to \mathcal{M}_1$ is a homeomorphism, and $|h(y)|=1$ for all $y\in \mathcal{M}_2$.
Since $\widetilde{\Phi}$ preserves the positive cone, we have $\widetilde{\Phi}(\mathds{1})\ge 0$, and hence $\widetilde{\Phi}(\mathds{1})(y) = h(y)\ge0$ for all $y\in \mathcal{M}_2$.
Here $\mathds{1}$ is the identity in $C(\mathcal{M}_1)$.
This implies $h \equiv 1$.
Therefore, the equality~\eqref{eq:BST} is written as 
$\widetilde{\Phi}(f)(y) = f(\eta(y))$ for all $f\in C(\mathcal{M}_1)$.
That is, $\widetilde{\Phi}(f) = f\circ\eta$ for all $f\in C(\mathcal{M}_1)$.
Then, we have $\widetilde{\Phi}(fg) = (fg)\circ\eta = (f\circ\eta)(g\circ\eta) = \widetilde{\Phi}(f)\widetilde{\Phi}(g)$, so $\widetilde{\Phi}$ is multiplicative.
Also $\widetilde{\Phi}(f^*)=\overline{f}\circ\eta = \overline{f\circ\eta} = \widetilde{\Phi}(f)^*$, the conjugation is preserved.
Therefore, $\widetilde{\Phi}$ is a $*$--isomorphism.
\end{proof}

For two ordered Banach spaces $W^{k,p}(\Omega_1)$ and $W^{k,p}(\Omega_2)$ having the same differentiability order $k\in\mathbb{N}$, we establish the following lemma.
\begin{lem}\label{lem:2.5}
    Let $L^p(\mu_1)$ and $L^p(\mu_2)$ be obtained from $W^{k,p}(\Omega_1)$ and $W^{k,p}(\Omega_2)$ by Proposition~\ref{prop:2.1}, respectively.
    Suppose $\widetilde{\Phi}:L^p(\mu_1)\to L^p(\mu_2)$ is arbitrary complex--linear isometric order isomorphism. 
    Then $\widetilde{\Phi}$ determines a unique permutation $\pi$ of $\{0,\ldots,k-1\}$ such that $\widetilde{\Phi}(u)(\ast_2^{(j)}) = u(\ast_1^{(\pi(j))})$ for all $u\in L^p(\mu_1)$.
    Moreover, there is a uniquely determined complex--linear isometric order isomorphism $\Lambda:L^p(\Omega_1)\to L^p(\Omega_2)$ .
\end{lem}
\begin{proof}
By Proposition~\ref{prop:2.2}, $\{\chi_{\ast_i^{(j)}}\}$ are the norm one lattice atoms in $L^p(\mu_i)$.
Since an isometric order isomorphism maps norm one lattice atoms onto norm one lattice atoms, it permutes the norm one atoms.
Consequently, there exists a unique permutation $\pi$ of $\{0,\ldots,k-1\}$ such that 
\begin{equation}\label{eq:atom-permutation}
\widetilde{\Phi}(\chi_{\{\ast_1^{\pi(j)}\}})=\chi_{\{\ast_2^{(j)}\}},\qquad j=0,\ldots,k-1.
\end{equation}
An order isomorphism preserves bands and their
disjoint complements, hence, for every $u\in L^p(\mu_1)$,
$$(\widetilde{\Phi}u)(\ast_2^{(j)})\chi_{\{\ast_2^{(j)}\}}=\widetilde{\Phi}\bigl(u(\ast_1^{(\pi(j))})\chi_{\{\ast_1^{\pi(j)}\}}\bigr)=u(\ast_1^{(\pi(j))}) \chi_{\{\ast_2^{(j)}\}},
$$
for each $j=0,\ldots,k-1$.
It follows that $\widetilde{\Phi}(u)(\ast_2^{(j)})=u(\ast_1^{(\pi(j))})$ for all $j=0,\ldots,k-1$.

Since $\widetilde{\Phi}(u)(\ast_2^{(j)})=u(\ast_1^{(\pi(j))})$ for all $u\in L^p(\mu_1)$ and $j=0,\ldots,k-1$, we have $\widetilde{\Phi}(L^p(\Omega_1))\subseteq L^p(\Omega_2)$.
Applying the same argument to $\widetilde{\Phi}^{-1}$, we conclude $\widetilde{\Phi}(L^p(\Omega_1))= L^p(\Omega_2)$.
Define $\Lambda:=\widetilde{\Phi}|_{L^p(\Omega_1)}$.
Since $\widetilde{\Phi}$ is a complex--linear isometric order isomorphism which preserves $\{\ast_i^{(0)},\ldots,\ast_i^{(k-1)}\}$, $\Lambda$ is a complex--linear isometric order isomorphism.
We conclude the required statement.
\end{proof}
We now give an affirmative answer to Problem~\ref{Tingley} for the spaces $W^{k,p}(\Omega)$.

\begin{thm}\label{thm:2.6}
    Suppose $1\le p \le \infty$, and $k\in\mathbb{N}$.
    Let $\Delta:S(W^{k,p}(\Omega_1))^+\to S(W^{k,p}(\Omega_2))^+$ be a surjective isometry.
    Then there exists a unique complex--linear isometric order isomorphism $\widetilde{\Delta}:W^{k,p}(\Omega_1) \to W^{k,p}(\Omega_2)$ such that $\widetilde{\Delta}|_{S(W^{k,p}(\Omega_1))^+} = \Delta$.
    More precisely,
    $$
    (\widetilde{\Delta}f)(x)=\sum_{j=0}^{k-1}f^{(\pi(j))}(x_{1})\frac{(x-x_{2})^j}{j!}+\frac{1}{(k-1)!}\int_{x_{2}}^{x}(x-t)^{k-1}\Lambda(f^{(k)})(t)\,dt,
    \qquad x\in\Omega_{2},
    $$
    where $\pi$ is a unique permutation of $\{0,\ldots,k-1\}$, and $\Lambda:L^p(\Omega_1)\to L^p(\Omega_2)$ is a unique complex--linear isometric order isomorphism, determined by $\Delta$.
\end{thm}
\begin{proof}
For $i=1,2$, following Proposition~\ref{prop:2.1}, we define a complex--linear isometric order isomorphism $J_i:W^{k,p}(\Omega_i)\to L^p(\mu_i)$. 
In particular, we have $J_i(S(W^{k,p}(\Omega_i))^+) = S(L^p(\mu_i))^+$.
Define a map $\Phi:=J_2\circ\Delta\circ J_1^{-1}: S(L^p(\mu_1))^+\to S(L^p(\mu_2))^+$.
Then $\Phi$ is a surjective isometry.
By Proposition~\ref{prop:2.3} along with Remark~\ref{rmk:2.4}, there exists a unique complex--linear isometric order isomorphism $\widetilde{\Phi}: L^p(\mu_1)\to L^p(\mu_2)$ such that $\widetilde{\Phi}|_{S(L^p(\mu_1))^+} = \Phi$.
Define $\widetilde{\Delta}:= J_2^{-1}\circ\widetilde{\Phi}\circ J_1:W^{k,p}(\Omega_1)\to W^{k,p}(\Omega_2)$.
Then $\widetilde{\Delta}$ is a complex--linear isometric order isomorphism as $J_1$, $J_2$ and $\widetilde{\Phi}$ are.

Now we verify $\widetilde{\Delta}|_{S(W^{k,p}(\Omega_1))^+} = \Delta$.
For $f\in S(W^{k,p}(\Omega_1))^+$, since $J_1(f)\in S(L^p(\mu_1))^+$, and hence $\Phi$ agrees with $\widetilde{\Phi}$ at $J_1(f)$, we have $\widetilde{\Delta}f = (J_2^{-1}\widetilde{\Phi}J_1)f =
J_2^{-1}\widetilde{\Phi}(J_1f) = J_2^{-1}\Phi(J_1f) = (J_2^{-1}\Phi J_1)f = \Delta f$.
Therefore, $\widetilde{\Delta}|_{S(W^{k,p}(\Omega_1))^+} = \Delta$.

We next prove the asserted representation.
Let $f\in W^{k,p}(\Omega_1)$.
Since $J_2(\widetilde{\Delta}f)=\widetilde{\Phi}(J_1f)$, Lemma~\ref{lem:2.5} gives a unique permutation $\pi$ of $\{0,\ldots,k-1\}$ and an induced complex--linear isometric order isomorphism $\Lambda:L^p(\Omega_1)\to L^p(\Omega_2)$.
Hence
$$
(\widetilde{\Delta}f)^{(j)}(x_2)=\bigl(J_2(\widetilde{\Delta}f)\bigr)(\ast_2^{(j)})=\widetilde{\Phi}(J_1f)(\ast_2^{(j)})=(J_1f)(\ast_1^{(\pi(j))})=f^{(\pi(j))}(x_1),
$$
for every $j=0,\ldots,k-1$.
Applying the equality~\eqref{eq:Sobolev} at the base point $x_2$, we obtain
$$
(\widetilde{\Delta}f)(x)=\sum_{j=0}^{k-1}(\widetilde{\Delta}f)^{(j)}(x_2)\frac{(x-x_2)^j}{j!}+\frac{1}{(k-1)!}\int_{x_2}^{x}(x-t)^{k-1}(\widetilde{\Delta}f)^{(k)}(t)\,dt.
$$
Substituting $\pi$ and $\Lambda$ into the preceding identities yields
$$(\widetilde{\Delta}f)(x)=\sum_{j=0}^{k-1}f^{(\pi(j))}(x_1)\frac{(x-x_2)^j}{j!}+\frac{1}{(k-1)!}\int_{x_2}^{x}(x-t)^{k-1}\Lambda(f^{(k)})(t)\,dt,\qquad x\in\Omega_2.$$

Finally, we show the uniqueness.
Suppose that there is another complex--linear isometric order isomorphism $\overline{\Delta}:W^{k,p}(\Omega_1)\to W^{k,p}(\Omega_2)$ such that $\overline{\Delta}|_{S(W^{k,p}(\Omega_1))^+}=\Delta$.
Then $\overline{\Phi}:=J_2\circ\overline{\Delta}\circ J_1^{-1}:L^p(\mu_1)\to L^p(\mu_2)$ is a complex--linear isometric order isomorphism such that the restriction $\overline{\Phi}|_{S(L^p(\mu_1))^+}$ is $\Phi$.
By the uniqueness of Proposition~\ref{prop:2.3} and Remark~\ref{rmk:2.4}, $\widetilde{\Phi} = \overline{\Phi}=J_2\circ\overline{\Delta}\circ J_1^{-1}$.
Therefore, $\widetilde{\Delta}=J_2^{-1}\circ\widetilde{\Phi}\circ J_1=\overline{\Delta}$.
\end{proof}

\begin{rmk}\label{rmk:2.7}
The base points used in the norms of $W^{k,p}(\Omega_1)$ and $W^{k,p}(\Omega_2)$ are fixed but not specific.
Different choices of base points give rise to complex--linearly isometric and order--isomorphic ordered Banach spaces.
Hence the conclusion of Theorem~\ref{thm:2.6} does not depend on the particular choice of base points.   
\end{rmk}

It is natural to ask whether the preceding results admit an analogue for Sobolev spaces on higher--dimensional domains.
Such an extension, however, does not follow directly from the present argument.
For an interval, the mapping $f \mapsto\bigl(f(x_{0}),\ldots,f^{(k-1)}(x_{0}),f^{(k)}\bigr)$ identifies $W^{k,p}$ with an $L^{p}$--space augmented by finitely many atomic coordinates.
For a Lipschitz domain in $\mathbb{R}^{d}$, $d\ge2$, the weak derivatives of order $k$ are constrained by compatibility relations among mixed derivatives, while lower--order point evaluations are not generally available.
Consequently, the $L^{p}$--space representation used in the present proof does not extend directly to the higher--dimensional setting.
Establishing an analogous result would therefore require a different representation of the ordered Sobolev space and its positive unit sphere.

\section{The differentiability order as an order--geometric invariant}
Having established unique extension for arbitrary surjective isometries, we now determine when a surjective isometry can exist between the positive unit spheres corresponding to possibly different differentiability orders.

\begin{thm}\label{thm:2.8}
Let $k_1,k_2\in\mathbb{N}$.
For each $1 \leq p \leq \infty$, the following statements are equivalent:
\begin{enumerate}
    \item $k_1=k_2$.
    \item There exists a complex--linear isometric order isomorphism 
    $$\widetilde{\Delta}\colon W^{k_1,p}(\Omega_1)\to W^{k_2,p}(\Omega_2).$$
    \item There exists a surjective isometry  
    $\Delta\colon S(W^{k_1,p}(\Omega_1))^+\to S(W^{k_2,p}(\Omega_2))^+.$
    \item There exists a surjective phase--isometry  
    $T\colon S(W^{k_1,p}(\Omega_1))^+\to S(W^{k_2,p}(\Omega_2))^+.$
\end{enumerate}
Additionally, if $1<p<\infty$,
    \begin{enumerate}
        \item[(5)] There exists a surjective norm--additive map  
    $\varphi\colon S(W^{k_1,p}(\Omega_1))^+\to S(W^{k_2,p}(\Omega_2))^+.$ 
    \end{enumerate}
In particular, every map in (3) or (4) admits a
unique extension as in (2) with the form expressed in Theorem~\ref{thm:2.6}.
If $1<p<\infty$, the same conclusion holds for every map in (5).
\end{thm}
\begin{proof}
$(3)\Rightarrow(2)$ 
For $i=1,2$, let $J_i:W^{k_i,p}(\Omega_i)\to L^p(\mu_i)$ be the map constructed in Proposition~\ref{prop:2.1}.
Then $\Phi:= J_2\circ\Delta\circ J_1^{-1}:S(L^p(\mu_1))^+\to S(L^p(\mu_2))^+$ is a surjective isometry.
By Proposition~\ref{prop:2.3} along with Remark~\ref{rmk:2.4}, there exists a unique complex--linear isometric order--isomorphic extension $\widetilde{\Phi}:L^p(\mu_1)\to L^p(\mu_2)$.
Thus, there is a complex--linear isometric order isomorphism $\widetilde{\Delta}:= J_2^{-1}\circ\widetilde{\Phi}\circ J_1: W^{k_1,p}(\Omega_1)\to W^{k_2,p}(\Omega_2)$.

$(2)\Rightarrow(1)$
For each $i=1,2$, let $J_i$ be the complex--linear isometric order isomorphism from Proposition~\ref{prop:2.1}, and put $\widetilde{\Phi}=J_2\circ\widetilde{\Delta}\circ J_1^{-1}\colon L^p(\mu_1)\to L^p(\mu_2)$.
Then $\widetilde{\Phi}$ is a complex--linear isometric order isomorphism.
By Proposition~\ref{prop:2.2},  $\widetilde{\Phi}$ induces a bijection
$$
\left\{
\chi_{\{\ast_1^{(j)}\}}:j=0,\ldots,k_1-1
\right\}
\to
\left\{
\chi_{\{\ast_2^{(j)}\}}:j=0,\ldots,k_2-1
\right\}.
$$
These two finite sets have the same cardinality, and hence $k_1=k_2$.

$(1)\Rightarrow (2)$ Suppose that $k_1=k_2=:k$. Write $\Omega_i=(a_i,b_i)$ and $\ell_i=b_i-a_i$ for $i=1,2$.
Define an affine bijection
$$
\tau\colon\Omega_2\longrightarrow\Omega_1,\qquad\text{by}\qquad
\tau(x)=a_1+\frac{\ell_1}{\ell_2}(x-a_2),\quad x\in\Omega_2.
$$
For each $g\in L^p(\Omega_1)$, define
\begin{align*}
(\Lambda g)(x)=
\begin{cases}
(\frac{\ell_1}{\ell_2})^{1/p}g(\tau(x))\quad&(1\le p<\infty)\\
g(\tau(x))\quad&(p=\infty)
\end{cases}
,\qquad x\in \Omega_2.
\end{align*}
Then $\Lambda\colon L^p(\Omega_1)\to L^p(\Omega_2)$ is a complex--linear isometric order isomorphism.
In particular, if $1\leq p<\infty$, then a change of variables gives
$$
\|\Lambda g\|_p^p=\int_{\Omega_2}
\frac{\ell_1}{\ell_2}\,|g(\tau(x))|^p\,dx=\int_{\Omega_1}|g(t)|^p\,dt=\|g\|_p^p.
$$
The case $p=\infty$ is immediate.
Moreover, both $\Lambda$ and $\Lambda^{-1}$ preserve positivity.

Let $(X_i,\mathfrak{A}_i,\mu_i)$ be the measure spaces associated with $W^{k,p}(\Omega_i)$ obtained by Proposition~\ref{prop:2.1}. 
Define
$$
\widetilde{\Phi}\colon L^p(\mu_1)\to L^p(\mu_2),\qquad\text{by}\qquad
\widetilde{\Phi}
\left(
\sum_{j=0}^{k-1}\alpha_j\chi_{\{\ast_1^{(j)}\}}+g\chi_{\Omega_1}
\right) 
= \sum_{j=0}^{k-1}\alpha_j\chi_{\{\ast_2^{(j)}\}}+(\Lambda g)\chi_{\Omega_2}.
$$
It follows directly from the definition of the $L^p$-norm on the atomic and atomless parts that $\widetilde{\Phi}$ is a complex--linear isometric order isomorphism.
Define $\widetilde{\Delta}=J_2^{-1}\circ\widetilde{\Phi}\circ J_1
\colon W^{k,p}(\Omega_1)\to W^{k,p}(\Omega_2)$.
Then $\widetilde{\Delta}$ is a complex--linear isometric order isomorphism, as $J_1,J_2$ and $\widetilde{\Phi}$ are.
This proves the equivalence of $(1)$ and $(2)$.
As a consequence, we assume $k_1=k_2$ in (2).

$(2)\Rightarrow(3)$ This follows directly from Definition~\ref{dfn:2.1}(3).

$(2)\Rightarrow (4)$ 
For every $f,g\in S(W^{k,p}(\Omega_1))^+$ and each $\alpha\in\mathbb{T}$, we have
\begin{align*}
\|\widetilde{\Delta}(f) - \alpha\widetilde{\Delta}(g)\|_{k,p}= \|\widetilde{\Delta}(f-\alpha g)\|_{k,p}=\|f-\alpha g\|_{k,p}.
\end{align*}
Thus, 
$$\big\{\|\widetilde{\Delta}(f) - \alpha\widetilde{\Delta}(g)\|_{k,p}:\alpha\in\mathbb{T}\big\} = \big\{\|f-\alpha g\|_{k,p}:\alpha\in\mathbb{T}\big\}.$$
Therefore, we conclude that $\Delta:=\widetilde{\Delta}|_{S(W^{k,p}(\Omega_1))^+}$ is a phase--isometry.

$(4)\Rightarrow (3)$
For $f$ and $g$ in $S(W^{k_1,p}(\Omega_1))^+$, we have the following identity
$$|f^{(j)}(x_1) - \alpha g^{(j)}(x_1)|^2 - |f^{(j)}(x_1) - g^{(j)}(x_1)|^2 = 2f^{(j)}(x_1)g^{(j)}(x_1)(1-\text{Re}(\alpha)),$$
for all $\alpha\in\mathbb{T}$ and $j=0,\ldots,k_1-1$.
Since $f^{(j)}(x_1)\ge0$ and $g^{(j)}(x_1)\ge 0$, as well as $\text{Re}(\alpha)\le 1$, the above identity is non--negative.
Hence 
\begin{equation}\label{eq:4.1}
|f^{(j)}(x_1) - g^{(j)}(x_1)|\le|f^{(j)}(x_1) - \alpha g^{(j)}(x_1)|,\qquad j=0,\ldots,k_1-1,\quad\alpha\in\mathbb{T}.
\end{equation}
The same argument applies to the $k_1$--th weak derivatives.
We have $|f^{(k_1)}(t) - g^{(k_1)}(t)|\le|f^{(k_1)}(t) - \alpha g^{(k_1)}(t)|$ for almost every $t\in\Omega_1$.
Thus, 
\begin{equation}\label{eq:4.2}
\|f^{(k_1)} - g^{(k_1)}\|_p\le \|f^{(k_1)} - \alpha g^{(k_1)}\|_p.
\end{equation}
Substituting inequalities~\eqref{eq:4.1} and \eqref{eq:4.2} into the definition of the $W^{k,p}$--norm~\eqref{W_norm}, we have $\|f - g\|_{k_1,p}\le \|f-\alpha g\|_{k_1,p}$ for each $\alpha\in\mathbb{T}$.
Similarly, we have $\|T(f) - T(g)\|_{k_2,p}\le \|T(f)-\alpha T(g)\|_{k_2,p}$ for such $f$ and $g$.
Then
\begin{align*}
    \|T(f) - T(g)\|_{k_2,p}&=\min_{\alpha\in\mathbb{T}}\|T(f) - \alpha T(g)\|_{k_2,p}\\
    &=\min_{\alpha\in\mathbb{T}}\|f-\alpha g\|_{k_1,p}\\
    &=\|f - g\|_{k_1,p}.
\end{align*}
The second equality above follows from the fact that $T$ is a phase--isometry.
Note that both maps $\alpha\mapsto\|f-\alpha g\|_{k_1,p}$ and $\alpha\mapsto\|T(f)-\alpha T(g)\|_{k_2,p}$ are both continuous on the compact set $\mathbb{T}$, so their minima are attained.
Thus, $T$ is a surjective isometry.

We have thus proved that (1), (2), (3), and (4) are equivalent.
Now consider $1<p<\infty$, let $\varphi$ be a map in $(5)$.
First, we prove a surjective norm--additive map is necessarily bijective.
Recall that $L^p$--spaces are strictly convex for each $1<p<\infty$, and hence so is $W^{k_i,p}(\Omega_i)$ by Proposition~\ref{prop:2.1}.
Assume $f, g\in S(W^{k_1,p}(\Omega_1))^+$ with $\varphi(f)=\varphi(g)$, then we have
\begin{align*}
\|f+g\|_{k_1,p} = \|\varphi(f)+\varphi(g)\|_{k_2,p}=\|2\varphi(f)\|_{k_2,p}=2=\|f\|_{k_1,p} + \|g\|_{k_1,p},
\end{align*}
where the last equality follows from $f,g$ being in the positive unit sphere.
However, the space $W^{k_1,p}(\Omega_1)$ is strictly convex.
The equality $\|f+g\|_{k_1,p} = \|f\|_{k_1,p}+\|g\|_{k_1,p}$ forces $f=g$, and hence $\varphi$ is injective.

$(5)\Rightarrow (2)$
Now define $\Phi:=J_2\circ\varphi\circ J_1^{-1}: S(L^p(\mu_1))^+\to S(L^p(\mu_2))^+$, we have
\begin{align*}
    \|\Phi(f)+\Phi(g)\|_{L^p(\mu_2)} &= \|J_2\big(\varphi J_1^{-1}(f) + \varphi J_1^{-1}(g)\big)\|_{L^p(\mu_2)}\\
    &= \|\varphi J_1^{-1}(f) + \varphi J_1^{-1}(g)\|_{k_2,p}\\
    &= \|J_1^{-1}(f) + J_1^{-1}(g)\|_{k_1,p}\\
    &= \|J_1^{-1}(f+g)\|_{k_1,p}\\
    &= \|f+g\|_{L^p(\mu_1)},
\end{align*}
for each $f,g\in S(L^p(\mu_1))^+$.
Hence $\Phi$ is a bijective norm--additive map.
By~\cite[Theorem 2.6]{HDL22}, there exists a real--linear isometric order isomorphism $L_\mathbb{R}^p(\mu_1)\to L_\mathbb{R}^p(\mu_2)$, and hence its unique complexification $\widetilde{\Phi}:L^p(\mu_1)\to L^p(\mu_2)$.
(Notice that $\mu_i(X_i) = m(\Omega_i) + k_i<\infty$, and hence a finite measure.)
Then, the map $\widetilde{\varphi}:=J_2^{-1}\circ\widetilde{\Phi}\circ J_1: W^{k_1,p}(\Omega_1)\to W^{k_2,p}(\Omega_2)$ is a complex--linear isometric order isomorphism.
Uniqueness follows immediately.

$(4)\Rightarrow (5)$ The proof is essentially the same as the argument in $(4)\Rightarrow (3)$ with the following equalities replacing the minima argument:
\begin{align*}
\|T(f)+T(g)\|_{k,p} = \max_{\alpha\in\mathbb{T}}\|T(f)-\alpha T(g)\|_{k,p} =\max_{\alpha\in\mathbb{T}}\|f-\alpha g\|_{k,p} =\|f+g\|_{k,p}.
\end{align*}
And also the following identity, 
$$|f^{(j)}(x_1) + g^{(j)}(x_1)|^2 - |f^{(j)}(x_1) - \alpha g^{(j)}(x_1)|^2 = 2f^{(j)}(x_1)g^{(j)}(x_1)(1+\text{Re}(\alpha)).$$
Note that (1) holds in (4).
Therefore, the required equivalence statements are proved.
\end{proof}
\begin{rmk}\label{rmk:2.9}
The following observations follow from Theorem~\ref{thm:2.8}.
\begin{enumerate}
    \item The first four statements imply (5) for each $1\le p\le\infty$.
However, in the converse direction, two endpoints $p=1$ and $p=\infty$ are excluded in the current argument.
    \item The implications from statement (2) to statements (3), (4), and (5) are valid for arbitrary ordered Banach spaces.
\end{enumerate}
\end{rmk}
The preceding result shows that the Sobolev order is determined by the metric structure of the positive unit sphere.
This phenomenon arises because the anchored order selects a metric subset that records the atomic coordinates rather than merely the linear structure of the underlying Sobolev space.
Indeed, when $p=2$, the anchored Sobolev space $W^{k,2}(\Omega)$ is linearly isometric to
$\mathbb{C}^{k}\oplus_{2}L^{2}(\Omega)$, and hence, as an infinite-dimensional separable Hilbert space, it is linearly isometric to $W^{\ell,2}(\Omega')$ for every $\ell\geq 1$.
Nevertheless, the positive unit spheres of $W^{k,2}(\Omega)$ and $W^{\ell,2}(\Omega')$ cannot be surjectively isometric unless $k=\ell$.
Thus, the positive unit sphere retains the number of atomic coordinates arising from the order $\ge_{k,p}$.

\section*{Acknowledgments}
The author was supported by the National Science and Technology Council, Taiwan (NSTC), under Grant No. 115--2917--I--110--009.

\begin{comment}

\end{comment}

\end{document}